\documentclass[11pt]{amsart}
\usepackage{xcolor}

\include{package}
\definecolor{grn}{rgb}{0,0.6,0}
\definecolor{mrn}{rgb}{0.3,0,0}
\definecolor{blue}{rgb}{0,0,0.7}
\definecolor{Mygray}{rgb}{0.75,0.75,0.75}
\definecolor{auburn}{rgb}{0.43, 0.21, 0.1}
\definecolor{britishracinggreen}{rgb}{0.0, 0.26, 0.15}
\definecolor{taupe}{rgb}{0.28, 0.24, 0.2}
\usepackage{amsmath,amssymb}
\newtheorem{theorem}{Theorem}

\theoremstyle{definition}
\newtheorem{defn}{Definition}
\newtheorem{question}{Question}
\newtheorem{rmk}{Remark}

\usepackage{latexsym,amssymb}
\usepackage{ textcomp }

\begin{document}
\baselineskip=14.5pt
\title[Denseness of direction and generalized direction sets]{On denseness of certain direction and generalized direction sets}

\author{Deepa Antony, Rupam Barman and Jaitra Chattopadhyay}
\address[Deepa Antony, Rupam Barman and Jaitra Chattopadhyay]{Department of Mathematics, Indian Institute of Technology, Guwahati, Guwahati - 781039, Assam, India}
\email[Deepa Antony]{deepa172123009@iitg.ac.in}

\email[Rupam Barman]{rupam@iitg.ac.in}

\email[Jaitra Chattopadhyay]{jaitra@iitg.ac.in}

\begin{abstract}
Direction sets, recently introduced by Leonetti and Sanna, are generalization of ratio sets of subsets of positive integers. In this article, we generalize the notion of direction sets and define {\it $k$-generalized direction sets} and {\it distinct $k$-generalized direction sets} for subsets of positive integers. We prove a necessary condition for a subset of $\mathcal{S}^{k - 1} := \{\underline{x} \in [0,1]^{k} : ||\underline{x}|| = 1\}$ to be realized as the set of accumulation points of a distinct $k$-generalized direction set. We provide sufficient conditions for some particular subsets of positive integers so that the corresponding $k$-generalized direction sets are dense in $\mathcal{S}^{k - 1}$. We also consider the denseness properties of certain direction sets and give a partial answer to a question posed by Leonetti and Sanna. Finally we consider a similar question in the framework of an algebraic number field.
\end{abstract}

\renewcommand{\thefootnote}{}



%\footnote{\emph{We confirm that all the data are included in the article.}}

\renewcommand{\thefootnote}{\arabic{footnote}}
\setcounter{footnote}{0}

\maketitle

\section{Introduction and statements of results}

For a non-empty set $A \subseteq \mathbb{N}$, the {\it ratio set} of $A$ is defined by $R(A) := \{\frac{a}{b} \in \mathbb{Q} : a,b \in A\}$. One of the most fundamental results in real analysis, viz. $\mathbb{Q}$ is dense in $\mathbb{R}$, when rephrased in terms of ratio sets, reads as the ratio set of $\mathbb{N}$ is dense in $\mathbb{R}_{> 0}$. This reformulation of the denseness of $\mathbb{Q}$ in $\mathbb{R}$ has spurred a lot of research in recent times. In particular, the classification of subsets of $\mathbb{N}$ having dense ratio sets in $\mathbb{R}_{> 0}$ has been a central question of investigation. In what follows, we say that $A$ is {\it fractionally dense} in $\mathbb{R}_{> 0}$ if $R(A)$ is dense in $\mathbb{R}_{> 0}$.

\smallskip

One of the most natural choices for $A$ is the set $\mathbb{P}$ of prime numbers and it is known to be fractionally dense (cf. \cite{hs}, \cite{Salat}). Several generalizations of this result have been proven over the years and several interesting subsets of natural numbers have been shown to be fractionally dense (cf. \cite{gems} - \cite{Bukor-Toth 2}, \cite{dense-Gauss}, \cite{Dio} - \cite{hs}, \cite{Salat} - \cite{Salat3}, \cite{ps} - \cite{Toth2}). In \cite{CRS}, \cite{dense-Gauss} and \cite{Sittinger}, analogous questions have been dealt with in the set up of algebraic number fields. Very recently, the denseness of ratio sets in the $p$-adic completion $\mathbb{Q}_{p}$ have also been considered (cf. \cite{AB}, \cite{ABM}, \cite{Luca1}, \cite{Luca2}, \cite{Sanna2}, \cite{Sanna}).

\smallskip

Very recently, Leonetti and Sanna \cite{BAMS} introduced the notion of {\it direction sets}, which generalizes the notion of ratio sets as follows. For an integer $k \geq 2$ and $\emptyset \neq A \subseteq \mathbb{N}$, they considered the following sets: 
\begin{align*}
\mathcal{S}^{k - 1} := \{\underline{x} \in [0,1]^{k} : ||\underline{x}|| = 1\}, ~ \mathcal{D}^{k}(A) := \{\rho(\underline{a}) : \underline{a} \in A^{k}\} ~ \mbox{ and } ~ \mathcal{D}^{\underline{k}}(A) := \{\rho(\underline{a}) : \underline{a} \in A^{\underline{k}}\},
\end{align*}
where $\rho : \mathbb{R}_{\geq 0}^{k} \to \mathcal{S}^{k - 1}$ is the map defined by $\rho(\underline{x}) = \frac{\underline{x}}{||\underline{x}||}$ and $A^{\underline{k}} = \{\underline{a} \in A^{k} : a_{i} \neq a_{j} \mbox{ for all } i \neq j\}$. The sets $\mathcal{D}^{k}(A)$ and $\mathcal{D}^{\underline{k}}(A)$ are called the $k$-direction sets of $A$. We note that, for $k = 2$, we can identify $\mathcal{S}^{1}$ with $[0,+\infty]$ via a bijective map and thus the question of denseness in $\mathbb{R}_{> 0}$ can be translated into that in $\mathcal{S}^{1}$. Therefore, direction sets are indeed generalizations of ratio sets. Leonetti and Sanna \cite[Theorem 1.2]{BAMS} proved a necessary and sufficient criterion that determines whether a set $X \subseteq \mathcal{S}^{k - 1}$ can be realized as the set of accumulation points of $\mathcal{D}^{\underline{k}}(A)$ for some $A \subseteq \mathbb{N}$. Moreover, they proved a sufficient condition (cf. \cite[Theorem 1.5]{BAMS}) that asserts whether $\mathcal{D}^{k}(A)$ is dense in $\mathcal{S}^{k - 1}$. 

%Their result is as follows.
%
%\begin{proposition} \cite[Theorem 1.5]{BAMS}\label{propn-1}
%Let $A \subseteq \mathbb{N}$ be such that there is a sequence $\{a_{n}\}_{n = 1}^{\infty} \subseteq A$ with $\displaystyle\lim_{n \to \infty}\frac{a_{n}}{a_{n + 1}} = 1$. Then $\mathcal{D}^{k}(A)$ is dense in $\mathcal{S}^{k - 1}$.
%\end{proposition}

In this article, we further generalize the notion of direction sets and introduce generalized $k$-direction sets as follows.

\begin{defn}\label{gen-defn}
Let $k \geq 2$ be an integer and let $U_{1},\ldots,U_{k}$ be non-empty subsets of $\mathbb{N}$. We define the $k$-generalized direction set for the $k$-tuple $(U_{1},\ldots,U_{k})$ to be $\mathcal{D}^{k}(U_{1},\ldots,U_{k}) := \{\rho(u_{1},\ldots,u_{k}) : u_{j} \in U_{j} \mbox{ for } j = 1,\ldots,k\}$. Also, we define the distinct $k$-generalized direction set to be $\mathcal{D}^{\underline{k}}(U_{1},\ldots,U_{k}) := \{\rho(u_{1},\ldots,u_{k}) : u_{j} \in U_{j} \mbox{ for } j = 1,\ldots,k \mbox{ and } u_{i} \neq u_{j} \mbox{ for all } i \neq j\}$.
\end{defn}
Our first theorem is an analogue of Theorem 1.2 of \cite{BAMS} for distinct $k$-generalized direction sets. For any set $X \subseteq \mathcal{S}^{k - 1}$, we denote by $X^{\prime}$ the set of accumulation points of $X$. Also, we denote by $S_{k}$ the symmetric group on $k$ elements $\{1,\ldots,k \}$. For a permutation $\pi \in S_k$, we define $\pi(x_1,\dots,x_k):=(x_{\pi(1)},\dots,x_{\pi(k)})$ for all $\underline{x}=(x_1,\dots,x_k)$ in $\mathcal{S}^{k-1}.$ Also, for any subset $I$ of $\{1,\dots,k\}$, we define $\rho_I(\underline{x}):=\rho(\underline{y})$ where $\underline{y}=(y_1,\dots,y_k)$ is defined as $y_i:=x_i$ if $i\in I$ and the other coordinates as $0$. We say that $I$ {\it meets} $\underline{x}$ if $x_i\neq0$ for some $i \in I$. We state our first theorem as follows.
\begin{theorem}\label{THE-BIG-MAIN-TH}
Let $k \geq 2$ be an integer.  For subsets $U_{1},\ldots,U_{k}$ of $\mathbb{N}$, let $X=\mathcal{D}^{\underline{k}}(U_1,\dots,U_k)^{\prime}$. Then, we have:
\begin{enumerate}
    \item[(i)] $X$ is a closed subset of $\mathcal{S}^{k - 1}$.
    \item[(ii)] If $U_{i_1}=\dots=U_{i_m}$ for some $\{i_1,\dots,i_m\}\subseteq\{1,\dots,k\}$, then  for $\pi\in S_k $ with $\pi(j)=j$ for all $j\notin \{i_1,\dots,i_m\}$, we have $\pi(\underline{x})\in X$ for every $\underline{x}\in X$.
    \item[(iii)] If $|U_{i}| \geq k$ for each $i \in \{1,\ldots,k\}$, then for every $I\subseteq\{1,\dots,k\}$ that meets $\underline{x}$, we have $\rho_I(\underline{x})\in X$.
\end{enumerate}
\end{theorem} 

We recall that for a non-empty set $A \subseteq \mathbb{N}$, the natural density of $A$ is defined as $d(A) := \displaystyle\lim_{X \to \infty}\frac{\#\{n \in A : n \leq X\}}{X}$, provided the limit exists. The next theorem provides a sufficient condition for $\mathcal{D}^{k}(U_1,\dots, U_k)$ to be dense in $\mathcal{S}^{k - 1}$.

\begin{theorem}\label{TH-1}
Let $k \geq 2$ be an integer and let $U_1,\dots, U_k\subseteq \mathbb{N}$ be such that $d(U_i)$ exists and equals $\delta_i>0$ $ \mbox{ for all } i=1,\dots,k$. Assume that $\displaystyle\bigcap_{i=1}^{k}U_i$ is an infinite set. Then 
$\mathcal{D}^{k}(U_1,\dots, U_k)$ is dense in $\mathcal{S}^{k-1}$.
\end{theorem}

The next theorem extends Theorem 1.5 of \cite{BAMS}, which asserts that if for a set $A \subseteq \mathbb{N}$, there exists an increasing sequence $\{a_{n}\}_{n = 1}^{\infty} \subseteq A$ with $\displaystyle\lim_{n \to \infty} \frac{a_{n}}{a_{n + 1}} = 1$, then $\mathcal{D}^{k}(A)$ is dense in $\mathcal{S}^{k - 1}$. We generalize this for $\mathcal{D}^{k}(U_1,\dots, U_k)$ as follows.

\begin{theorem}\label{th-2}
Let $k \geq 2$ be an integer and let $U_1, U_2,\dots, U_k$ be non-empty subsets of $\mathbb{N}$. If there exist increasing sequences $u_i^{(n)}\subseteq U_i$ for all $i \in \{1,\ldots,k\}$ such that $\displaystyle\lim_{n \to \infty}\frac{u_{i}^{(n - 1)}}{u_i^{(n)}}=1$, then $\mathcal{D}^{k}(U_1,\dots, U_k)$ is dense in $\mathcal{S}^{k-1}$.
\end{theorem}

\begin{rmk}
For an integer $k \geq 2$ and for each $i \in \{1,\ldots,k\}$, let $a_{i}$ and $m_{i}$ be integers with $\gcd(a_{i},m_{i}) = 1$. Let $\mathbb{P}_{m_{i}} := \{p \in \mathbb{P} : p \equiv a_{i} \pmod{m_{i}}\}$. For $U_{i} = \mathbb{P}_{m_{i}}$, using Dirichlet's theorem for primes in arithmetic progressions, we see that the hypotheses of Theorem \ref{th-2} are satisfied. Therefore, $\mathcal{D}^{k}(\mathbb{P}_{m_{1}},\ldots,\mathbb{P}_{m_{k}})$ is dense in $\mathcal{S}^{k - 1}$.
\end{rmk}

\begin{theorem}\label{TH--3}
Let $k \geq 2$ be an integer and for each $i \in \{1,\ldots,k\}$, let $f_{i}(X_1,\dots, X_m)\in \mathbb{Z}[X_1,\dots, X_m]$ be polynomials of total degree $d_{i}$ such that the sum of the coefficients of degree $d_{i}$ terms is positive. Let $U_{i}:=\{f_{i}(n_1,\dots,n_m)|(n_1,\dots,n_m)\in \mathbb{N}^m\}\cap \mathbb{N}.$ Then $\mathcal{D}^k(U_{1},\ldots,U_{k})$ is dense in $\mathcal{S}^{k-1}.$
\end{theorem}

In \cite{Toth-Salat}, it is proven that there is a $3$-partition of $\mathbb{N} = A \cup B \cup C$, such that none of $R(A), R(B)$ and $R(C)$ is dense in $\mathbb{R}_{> 0}$. That is, none of $\mathcal{D}^{2}(A), \mathcal{D}^{2}(B)$ and $\mathcal{D}^{2}(C)$ is dense in $\mathcal{S}^{1}$. In \cite{BAMS}, Leonetti and Sanna asked for a possible generalization of this result for $k \geq 3$ \cite[Question 1.9]{BAMS}. We give a partial answer to their question in the next theorem.

\begin{theorem}\label{partition}
Let $k \geq 3$ be an integer. Then there exists a $3$-partition $\mathbb{N} = A \cup B \cup C$ of $\mathbb{N}$ such that none of $\mathcal{D}^{k}(A), \mathcal{D}^{k}(B)$ or $\mathcal{D}^{k}(C)$ is dense in $\mathcal{S}^{k - 1}$.
\end{theorem}

\begin{rmk}
In view of Theorem \ref{partition}, it remains to be seen whether for a $2$-partition $\mathbb{N} = A \cup B$, either $\mathcal{D}^{k}(A)$ or $\mathcal{D}^{k}(B)$ is dense in $\mathcal{S}^{k - 1}$ or not. We note that Theorem \ref{th-2} cannot be directly applied to address this issue. This can be seen by considering $A = \displaystyle\bigcup_{k = 0}^{\infty} [3^{k},2\cdot 3^{k})\cap \mathbb{N}$ and $B = \displaystyle\bigcup_{k = 0}^{\infty} [2\cdot 3^{k}, 3^{k + 1})\cap \mathbb{N}$. For, if $\{a_{n}\}_{n = 1}^{\infty} \subseteq A$ is an infinite sequence, then there are infinitely many indices $i$ for which $a_{i} \in [3^{k},2\cdot 3^{k})$ and $a_{i + 1} \in [3^{\ell},2\cdot 3^{\ell})$ for $k < \ell$. Then it follows that $\frac{a_{i}}{a_{i + 1}} < \frac{2\cdot 3^{k}}{3^{\ell}} \leq \frac{2}{3}$. Therefore, the elements of the sequence $\{\frac{a_{n}}{a_{n + 1}}\}_{n = 1}^{\infty}$ cannot get arbitrarily close to $1$. Similar argument works for $B$ as well. Thus there exist a $2$-partition of $\mathbb{N}$, none of which contains a sequence with the ratio of consecutive terms converging to $1$.
\end{rmk}

\smallskip

One of the interesting questions in the literature of fractionally dense sets is to look for sets $A \subseteq \mathbb{N}$ such that the ratio set $R(A)$ is dense in $\mathbb{R}_{>0}$ but $A$ contains no $3$-term arithmetic progressions. One such set is $A = \{2^{m} : m \geq 2\} \cup \{3^{n} : n \geq 2\}$, which is known to be fractionally dense in $\mathbb{R}_{>0}$ but $A$ contains no $3$-term arithmetic progressions (cf. \cite[Proposition 6]{gems}). In view of this, we may ask the following question.

\begin{question}\label{quest-1}
For an integer $k \geq 2$, does there exist a set $A \subseteq \mathbb{N}$ such that $A$ contains no $3$-term arithmetic progressions and $\mathcal{D}^{k}(A)$ is dense in $\mathcal{S}^{k - 1}$?

%For $A = \{2^{m} : m \geq 2\} \cup \{3^{n} : n \geq 2\}$, is $\mathcal{D}^{k}(A)$ is dense in $\mathcal{S}^{k - 1}$?
\end{question}

We answer Question \ref{quest-1} assertively in the following theorem.

\begin{theorem}\label{prop-2}
There exists a set $A \subseteq \mathbb{N}$ such that $A$ contains no $3$-term arithmetic progressions and $\mathcal{D}^{k}(A)$ is dense in $\mathcal{S}^{k - 1}$.
\end{theorem}

\begin{rmk}
We shall see in the proof of Theorem \ref{prop-2} that we can obtain infinitely many sets $A \subseteq \mathbb{N}$ having no arithmetic progression of length $3$ such that $\mathcal{D}^{k}(A)$ is dense in $\mathcal{S}^{k - 1}$.
\end{rmk}
Next, we discuss the denseness of some particular type of sets whose properties have been recently considered in \cite{kfk}. For an arithmetic function $f : \mathbb{N} \to \mathbb{N}$ and a positive real number $X$, let $f_{X} := \#\{n \leq X : n = kf(k) \mbox{ for some } k \in \mathbb{N}\}$. Keeping this notation, we state the results of \cite{kfk} as follows.

\begin{theorem} \cite{kfk}\label{kfk}
(i) Let $\omega (n) = \displaystyle\sum_{\substack
{p \mid n \\ p \in \mathbb{P}}}1$ be the prime divisor function. Then $$\omega_{X} = \frac{X}{\log \log X} + o\left(\frac{X}{\log \log X}\right).$$

\smallskip

(ii) Let $\phi (n) = \#\{1 \leq k \leq n : \gcd (k,n) = 1\}$ be the Euler's totient function. Then $$\phi_{X} = cX^{\frac{1}{2}} + o(X^{\frac{1}{2}}),$$ where $c = \displaystyle\prod_{p}\left(1 + \frac{1}{p(p - 1 + \sqrt{p^{2} - p})}\right) \sim 1.365\ldots$.
\end{theorem}

Now, we state our result as follows.

\begin{theorem}\label{kfk-TH}
Let $A = \{n\omega (n) : n \in \mathbb{N}\}$ and $B = \{n\phi (n) : n \in \mathbb{N}\}$. Then for any integer $k \geq 2$, we have that both $\mathcal{D}^{k}(A)$ and $\mathcal{D}^{k}(B)$ are dense in $\mathcal{S}^{k - 1}$.
\end{theorem}

\section{Proof of Theorems}
In this section, we prove our theorems. We first prove Theorem \ref{THE-BIG-MAIN-TH}.

\begin{proof}[Proof of Theorem \ref{THE-BIG-MAIN-TH}]
Since $X$ is the set of accumulation points of a subset of $\mathcal{S}^{k - 1}$, we immediately conclude that $X$ is closed and (i) is satisfied. 

\smallskip

Now, let $\underline{x}=(x_1,x_2,\dots,x_k)\in X=\mathcal{D}^{\underline{k}}(U_1,\dots,U_k)^{\prime}$. Then there exists a sequence $\rho(\underline{a}^{(n)})\in \mathcal{D}^{\underline{k}}(U_1,\dots,U_k)$ converging to $\underline{x}$ such that $\rho(\underline{a}^{(n)})\neq \underline{x}$ for infinitely many $n$, where $\underline{a}^{(n)} \in \displaystyle\prod_{i = 1}^{k}U_{i}$. For $\pi \in S_{k}$ with $\pi (j) = j$ for all $j \notin \{i_{1},\ldots i_{m}\}$, we consider $\underline{b}^{(n)}:=\pi(\underline{a}^{(n)})\in \mathcal{D}^{\underline{k}}(U_1,\dots,U_k)$. Then $\rho(\underline{b}^{(n)})$ converges to $\pi(\underline{x}).$ Consequently, we have $\pi(\underline{x}) \in X$ for every $\underline{x} \in X$ and thus (ii) is satisfied. 

\smallskip

Now, assume that $I$ is a non-empty subset of  $\{1,\dots,k\}$ that meets $\underline{x}$. We can consider a sub-sequence of $\underline{a}^{(n)}$ such that each $a_i^{(n)}$ is non-decreasing for each $i \in \{1,\ldots,k\}$. If $j\in \{1,\dots,k\}\setminus I$, then we can choose distinct $c_j\in U_j$ such that for sufficiently large positive integer $n_{0}$, a sequence $\underline{d}^{(n)}\in U_1\times\dots\times U_k$ with distinct coordinates can be defined for all $n \geq n_{0}$ with $d_i^{(n)}:=a_i^{(n)}$ for $i\in I$ and $d_i^{(n)}:=c_i$ for $i\notin I$. This choice is possible because of the assumption $|U_{i}| \geq k$ for each $i$. It then follows that $\rho(\underline{d}^{(n)})$ converges to $\rho_I(\underline{x})$. Thus (iii) holds. This completes the proof of Theorem \ref{THE-BIG-MAIN-TH}.
\end{proof}
\begin{proof}[Proof of Theorem \ref{TH-1}]
Let $\underline{x}\in (x_1,\dots, x_k)\in \mathcal{S}^{k-1}$ and let $I_i=(a_i,b_i)$ be open intervals such that $x_{i} \in I_{i}$ for each $i \in \{1,\ldots,k\}$. Then $\displaystyle\prod_{i=1}^{k} (a_i,b_i)\cap \mathcal{S}^{k-1}$ is a basic open set in $\mathcal{S}^{k - 1}$ containing $\underline{x}$. For a real number $X > 1$, let $U_i(X):=\#\{u_i\in U_i| u_i\leq X\}.$ By the hypothesis, we have that $\lim_{X \to \infty}\frac{U_i(X)}{X}=\delta_i>0$. This implies that $ U_i(X)=\delta_iX+o(X)$.
Therefore, \[\lim_{X\rightarrow\infty}\frac{U_i(a_iX)}{U_i(b_iX)}= \lim_{X\rightarrow\infty}\frac{\delta_ia_iX+o(a_iX)}{\delta_ib_iX+o(b_iX)}=\frac{a_i}{b_i}<1.\]
Thus for all sufficiently large real number $X$, there exists $u_i\in U_i$ such that $a_iX<u_i\leq b_iX$. That is, $ a_i<\frac{u_i}{X}\leq b_i.$
Since $\displaystyle\bigcap_{i=1}^{k}U_i$ is an infinite set, we can choose a large enough element $u \in \displaystyle\bigcap_{i=1}^{k}U_i$ such that $a_iu<u_i\leq b_iu$ for all $i=1,\dots,k$. This, in turn, implies that $\frac{u_i}{u}\in (a_i,b_i)$. Using the fact that $\rho(\underline{\alpha})=\frac{\underline{\alpha}}{\lVert \underline{\alpha}\rVert}$ is continuous function, we see that $\rho(u_1,\dots,u_k)\in \displaystyle\prod_{i=1}^kI_i\cap \mathcal{S}^{k-1}$. In other words,  $\mathcal{D}^k(U_1,\dots,U_k)$ is dense in $\mathcal{S}^{k-1}$.
\end{proof}
We next prove Theorem \ref{th-2} which extends Theorem 1.5 of \cite{BAMS}.
\begin{proof}[Proof of Theorem \ref{th-2}]
Let $\underline{x}=(x_1,\dots,x_k)\in \mathcal{S}^{k-1}$ with $x_i>0$ $\forall$ $i\in \{1,\dots,k\}$. We pick an integer $m$ such that $m>\frac{u_i^{(1)}}{\min\{x_1,\dots,x_k\}}$ $\forall$ $i\in \{1,\dots,k\}$. Then there exist integers $m_i$ for each $i\in\{1,\dots,k\}$ such that $u_{i}^{(m_i-1)}\leq mx_i< u_{i}^{(m_i)}$. That is,  $x_i<\frac{u_{i}^{(m_i)}}{m}\leq\frac{u_{i}^{(m_i)}}{u_{i}^{(m_i-1)}}x_i$. Since $m_i\rightarrow\infty$ as $m\rightarrow\infty$, it follows that $\displaystyle\lim_{m\rightarrow\infty}\frac{u_{i}^{(m_i)}}{m}=x_i$. Consequently, $\underline{u}=(u_{1}^{(m_1)},\dots,u_{k}^{(m_k)})$ converges to $\underline{x}$. Since $\rho $ is a continuous map, $\rho(\underline{u})$ converges to $\underline{x}$. Consequently, $\mathcal{D}^k(U_1,\dots,U_k)$ is dense in $\mathcal{S}^{k-1}$.
\end{proof}
\begin{proof}[Proof of Theorem \ref{TH--3}]
For a fixed integer $i \in \{1,\ldots,k\}$, we consider the polynomial $g_{i}(X)$ obtained by replacing all variables of $g_{i}$ by the variable $X$. We get, $g_{i}(X)=a_{d_{i}}X^{d_{i}}+a_{d_{i} - 1}X^{d_{i}-1}+\dots+a_0\in \mathbb{Z}[X]$. Since $a_{d_{i}} > 0$, we conclude that for a sufficiently large positive real number $X$,  we have $g_{i}(X) > 0$. Let $B_{i}:=\{g_{i}(n)|n\in \mathbb{N}\}\cap \mathbb{N}$. We have $\frac{g_{i}(X-1)}{g_{i}(X)}=\frac{a_{d_{i}}(X-1)^{d_{i}}+\dots+a_0}{a_{d_{i}}X^{d_{i}}+\dots+a_o}$ which tends to $1$ as $X$ tends to $\infty$. Also, since $g_{i}(X)$ is a polynomial in one variable, the sequence $\{g_{i}(n)\}_{n = 1}^{\infty}$ is eventually increasing. Therefore, by using Theorem \ref{th-2}, we obtain that $\mathcal{D}^k(B_{i})$ is dense in $\mathcal{S}^{k-1}.$ Since $B_{i} \subseteq U_{i}$, we conclude that $\mathcal{D}^k(U_{1},\ldots,U_{k})$ is dense in $\mathcal{S}^{k-1}$. 
\end{proof}
We now prove Theorem \ref{partition} which gives a partial answer to \cite[Question 1.9]{BAMS}.
\begin{proof}[Proof of Theorem \ref{partition}]
We consider the following three sets as in \cite{Toth-Salat} (see also \cite{gems}).
\begin{align*}A &:= \displaystyle\bigcup_{k = 0}^{\infty} [5^{k},2\cdot 5^{k})\cap \mathbb{N},\\
B &:= \displaystyle\bigcup_{k = 0}^{\infty} [2\cdot 5^{k},3\cdot 5^{k})\cap \mathbb{N},\\
C &:= \displaystyle\bigcup_{k = 0}^{\infty} [3\cdot 5^{k},5\cdot 5^{k})\cap \mathbb{N}.
\end{align*}
If $\mathcal{D}^{k}(A)$, $\mathcal{D}^{k}(B)$ or $\mathcal{D}^{k}(C)$ is dense in $\mathcal{S}^{k - 1}$, then by Theorem 1.4 of \cite{BAMS}, which states that if $\mathcal{D}^{k}(A)$ is dense in $\mathcal{S}^{k - 1}$ for some $A \subseteq \mathbb{N}$, then $\mathcal{D}^{k - 1}(A)$ is dense in $\mathcal{S}^{k - 2}$, we see inductively that $\mathcal{D}^{2}(A)$ (or $\mathcal{D}^{2}(B)$ or $\mathcal{D}^{2}(C)$) is dense in $\mathcal{S}^{1}$, which is false (cf. \cite[Proposition 3]{gems}). Therefore, we get a $3$-partition of $\mathbb{N}$ such that  none of $\mathcal{D}^{k}(A)$, $\mathcal{D}^{k}(B)$ or $\mathcal{D}^{k}(C)$ is dense in $\mathcal{S}^{k - 1}$. This completes the proof of Theorem \ref{partition}. 
\end{proof}
\begin{proof}[Proof of Theorem \ref{prop-2}]
In \cite{darmon}, it has been proven that the equation $x^{n} + y^{n} = 2z^{n}$ has no non-trivial solution in $\mathbb{Z}$ if $n \geq 3$. In other words, the set $A := \{m^{r}:  r, m \in \mathbb{Z}, r \geq 3\}$ does not contain any $3$-term arithmetic progressions. Since for a fixed value of $r \geq 3$, we have $\frac{m^{r}}{(m + 1)^{r}} \to 1$ as $m \to \infty$, by Theorem 1.5 of \cite{BAMS}, we conclude that $\mathcal{D}^{k}(A)$ is dense in $\mathcal{S}^{k - 1}$.
\end{proof}
\begin{proof}[Proof of Theorem \ref{kfk-TH}]
Let $\underline{x} = (x_{1},\ldots,x_{k}) \in \mathcal{S}^{k - 1}$ and let $\displaystyle\prod_{i = 1}^{k}(a_{i},b_{i})$ be a basic neighborhood of $\underline{x}$. Then by Theorem \ref{kfk}, we see that $$\displaystyle\lim_{X \to \infty} \frac{\omega_{a_{i}X}}{\omega_{b_{i}X}} = \displaystyle\lim_{X \to \infty}\frac{a_{i}X}{\log \log a_{i}X}\cdot \frac{\log \log b_{i}X}{b_{i}X} = \frac{a_{i}}{b_{i}} < 1 \mbox{ for all } i \mbox{ with } 1 \leq i \leq k.$$ 
Therefore, for sufficiently large $X$, there exists $\alpha_{i} \in A$ such that $a_{i}X < \alpha_{i} < b_{i}X$ for all $i$. That is, $\left(\frac{\alpha_{1}}{X},\ldots,\frac{\alpha_{k}}{X}\right) \in \displaystyle\prod_{i = 1}^{k}(a_{i},b_{i})$. Hence $\rho(\alpha_{1},\ldots,\alpha_{k}) = \rho\left(\frac{\alpha_{1}}{X},\ldots,\frac{\alpha_{k}}{X}\right) \in \displaystyle\prod_{i = 1}^{k}(a_{i},b_{i})$. Consequently, $\mathcal{D}^{k}(A)$ is dense in $\mathcal{S}^{k - 1}$.

\smallskip

Similarly, for $\mathcal{D}^{k}(B)$, we note that $$\displaystyle\lim_{X \to \infty} \frac{\phi_{a_{i}X}}{\phi_{b_{i}X}} =  \frac{\sqrt{a_{i}}}{\sqrt{b_{i}}} < 1 \mbox{ for all } i \mbox{ with } 1 \leq i \leq k$$ and thereafter it follows a similar line of argument. 
\end{proof}
\section{Concluding remarks : Case of algebraic number fields}
The ratio sets have been studied in the context of algebraic number fields in \cite{CRS}, \cite{dense-Gauss} and \cite{Sittinger}. It is interesting to extend the notion of direction sets in the set up of number fields and formulate analogous questions for the same.

\smallskip

Let $K \subsetneq \mathbb{R}$ be a number field of degree $d \geq 2$ and let $\mathcal{O}_{K}$ be its ring of integers. Let $\mathcal{O}_{K}^{0} := \{\alpha \in \mathcal{O}_{K} : {\rm{Tr}}_{K/\mathbb{Q}}(\alpha) = 0\}$ be the set of elements in $\mathcal{O}_{K}$ with trace $0$. Since $\mathcal{O}_{K}$ is a free $\mathbb{Z}$-module of rank $d$ and ${\rm{Tr}}$ is an additive group homomorphism from $\mathcal{O}_{K}$ to $\mathbb{Z}$, we see that $\mathcal{O}_{K} \cong \mathcal{O}_{K}^{0} \oplus \mathbb{Z}$. In particular, $\mathcal{O}_{K}^{0}$ is a free $\mathbb{Z}$-module of rank $d - 1$. Therefore, $\mathcal{O}_{K}^{0}$ itself is dense in $\mathbb{R}$ whenever $d \geq 3$. Also, for $d = 2$, we see that the ratio set of $\mathcal{O}_{K}^{0}$ is $\mathbb{Q}$. Consequently, the direction set of $\mathcal{O}_{K}^{0}$ is dense in $\mathcal{S}^{k - 1}$ for any integer $k \geq 2$.

\smallskip

We note that $\mathcal{O}_{K}^{0} \cap \mathbb{N} = \emptyset$. In view of this, we ask the following question.

\begin{question}
Let $d \geq 2$ and $k \geq 2$ be integers and let $K$ be a number field of degree $d$. Characterize the sets $\mathcal{A} \subseteq \mathcal{O}_{K}$ such that $\mathcal{A} \cap \mathbb{N}$ is finite and $\mathcal{D}^{k - 1}(\mathcal{A})$ is dense in $\mathcal{S}^{k - 1}$.
\end{question}

\bigskip

{\bf Acknowledgements.} We would like to thank IIT Guwahati for providing excellent facilities to carry out the research. The third author gratefully acknowledges the National Board of Higher Mathematics (NBHM) for the Post-Doctoral Fellowship (Order
No. 0204/16(12)/2020/R \& D-II/10925).

\end{document}